\documentclass[11pt]{article}

\usepackage{amssymb,amsbsy,amsmath,amscd,amsthm,amsfonts}
\usepackage{cite}
\usepackage[utf8]{inputenc}
\usepackage{enumerate,hyperref}
\usepackage{dsfont}

\newtheorem{theorem}{Theorem}

\newtheorem{lemma}{Lemma}
\newtheorem{remark}{Remark}

\newtheorem{corollary}{Corollary}

\newcommand{\R}{\mathbb R}
\newcommand{\PP}{\mathbb P}
\newcommand{\E}{\mathbb E}
\newcommand{\K}{\mathcal K_d}
\newcommand{\N}{\mathbb N}

\newcommand*\diff{\mathop{}\!\mathrm{d}}

\newcommand{\DS}{\displaystyle}

\title{\textbf{Uniform deviation and moment inequalities for random polytopes with general densities in arbitrary convex bodies}}
\author{Victor-Emmanuel Brunel\\
   Massachusetts Institute of Technology\\
   \texttt{vebrunel@mit.edu}}

\date{}

\begin{document}

\maketitle

\begin{abstract}
    We prove an exponential deviation inequality for the convex hull of a finite sample of i.i.d. random points with a density supported on an arbitrary convex body in $\R^d$, $d\geq 2$. When the density is uniform, our result yields rate optimal upper bounds for all the moments of the missing volume of the convex hull, uniformly over all convex bodies of $\R^d$: We make no restrictions on their volume, location in the space or smoothness of their boundary. After extending an identity due to Efron, we also prove upper bounds for the moments of the number of vertices of the random polytope. Surprisingly, these bounds do not depend on the underlying density and we prove that the growth rates that we obtain are tight in a certain sense.
\end{abstract}

\vspace{4mm}  

\section{Introduction}

    Probabilistic properties of random polytopes have been studied extensively in the literature in the last fifty years. Consider a collection of i.i.d. uniform random points in a convex body $K$ in $\R^d$. Their convex hull is a random polytope whose volume and number of vertices have been first analyzed in the seminal work of Rényi and Sulanke \cite{RenyiSulanke63,RenyiSulanke64}. They derived the asymptotics of the expected volume in the case $d=2$, when $K$ is either a polygon, with a given number of vertices, or a convex set with smooth boundary. More recently, considerable efforts were devoted to understanding the behavior of the expected volume. Several particular cases were investigated: For instance, when $K$ is a $d$-dimensional simple polytope\footnote{A $d$-dimensional simple polytope is a convex polytope such that each of its vertices is adjacent to exactly $d$ edges.} \cite{AffentrangerWieacker91}, a $d$-dimensional polytope \cite{BaranyBuchta93} or a $d$-dimensional Euclidean ball \cite{BuchtaMuller84}. Groemer \cite{Groemer74} (see also the references therein) proved that then $K$ has volume one, the expected volume of the random polytope is minimum when $K$ is an ellipsoid. B\'ar\'any and Larman \cite{BaranyLarman88} showed that if $K$ has volume one, then one minus the expected volume has the same asymptotic behavior as the volume of the $(1/n)$-wet part of $K$, defined as the union of all caps of $K$ (a cap being the intersection of $K$ with a half space) of volume at most $1/n$. Here, $n$ is the number of uniform random points in $K$. This remarkable result reduces the initial probabilistic problem to computation of such a deterministic volume; This purely analytical problem was then extensively studied. When $K$ has a smooth boundary, a key point was the introduction of the affine surface area, see \cite{SchuttWerner90,Schutt93}, and the volume of the $(1/n)$-wet part is of the order $n^{-2/(d+1)}$. When $K$ is a polytope, it is of a much smaller order, namely, $(\ln n)^{d-1}/n$ \cite{BaranyLarman88}; The expected volume is actually maximal when $K$ is a simple polytope \cite{BaranyLarman88}. As a conclusion, the expectation of the volume is now very well-understood, when the underlying distribution is uniform. Much less is known about its higher moments and the tails of its distribution. Using a jackknife inequality for symmetric functions, Reitzner \cite{Reitzner03} proved that the boundary of $K$ is smooth, the variance of the volume is bounded from above by $n^{-(d+3)/(d+1)}$, and he conjectured that this is the right order of magnitude for the variance. In addition, he proved that the second moment of the missing volume (i.e., the volume of $K$ minus the volume of the random polytope) is exactly of the order $n^{-4/(d+1)}$, with explicit constants that depend on the affine surface area of $K$.
    Using martingales inequalities, Vu \cite{Vu05} obtained deviation inequalities for arbitrary convex bodies of volume one, involving quantities such as the volume of the wet part, and derived precise deviation inequalities in the two important cases when $K$ is a polytope or has a smooth boundary. However, these inequalities involve constants which depend on $K$ but are not explicit. As a consequence, upper bounds on the moments of the missing volume are proved, again with non explicit constants depending on $K$: Let $K$ have volume one and $V_n$ stand for the missing volume, then there exist positive constants $\alpha, c$ and $\epsilon_0$ such that
		\begin{align}
			\PP\left[|V_n-\E[V_n]|\geq\sqrt{\lambda v}\right] & \leq 2e^{-\lambda/4}+e^{-c\epsilon n}, \label{Vu1}\\
			& \forall \epsilon\in \big(\alpha(\ln n)/n, \epsilon_0\big], \lambda\in \big(0,n|K(\epsilon)|\big], \nonumber
		\end{align}
		where $v=36ng(\epsilon)^2|K(\epsilon)|, g(\epsilon)=\sup\{|F| : F \mbox{ star-shaped} \subseteq K(\epsilon)\}$, and $K(\epsilon)$ is the $\epsilon$-wet part of $K$ defined in \cite{BaranyLarman88}.
		Moreover, if $K$ has a smooth boundary and volume one, there exist positive constants $c$ and $\alpha$, which depend on $K$, such that for any $\lambda\in \left(0,(\alpha/4)n^{-\frac{(d-1)(d+3)}{(d+1)(3d+5)}}\right]$, the following holds \cite{Vu05}:
\begin{equation} \label{DevVu}
    \PP\left[|V_n-\E[V_n]|\geq\sqrt{\alpha\lambda n^{-\frac{d+3}{d+1}}}\right]\leq 2\exp(-\lambda/4)+\exp\left(-cn^{\frac{d-1}{3d+5}}\right).
\end{equation}
    This inequality yields upper bounds on the variance and on the $q$-th moment of the missing volume, respectively of orders $n^{-(d+3)/(d+1)}$ and $n^{-2q/(d+1)}$, for $q>0$, for a \textit{smooth convex body $K$ of volume one}, up to constant factors that depend on $K$ in an unknown way. Note that these two inequalities proved by Vu remain true when $K$ has any positive volume, if $|V_n-\E[V_n]|$ is replaced by $|V_n-\E[V_n]|/|K|$. In our paper, we do not assume that the underlying distribution is uniform on $K$. We prove deviation inequalities and moment inequalities for a weighted missing volume, for general densities supported on the convex body $K$. In the uniform case, our results yield a deviation inequality which, unlike \eqref{Vu1} and \ref{DevVu}, which hold for a very small range of $\lambda$, captures the whole tail of the distribution of $V_n$. Our inequality is uniform over all convex bodies $K$, no matter their volume and boundary structure, and our constants do not depend on $K$. Our approach is based on a very simple covering number argument and is not bonded to the uniform distribution, which, to our best knowledge, makes our deviation inequalities completely new. 

In addition, we derive moment inequalities for the number of vertices of the random polytope. In the uniform case, we prove that the rates in our upper bounds are tight, uniformly on all convex bodies. As a consequence, we also prove that the growth of the moments of the number of vertices is the highest when the underlying density is uniform.

\section{Notation and statement of the problem}

    Let $d\geq 2$ be an integer. We denote by $|\cdot|$ the Lebesgue measure in $\R^d$, $\rho$ the Euclidean distance in $\R^d$, $B_d$ the unit Euclidean ball with center $0$, and $\beta_d$ its volume. \\

    If $G\subseteq\R^d$ and $\epsilon>0$, we denote by $G^\epsilon=\{x\in\R^d:\rho(x,G)\leq\epsilon\}$ the closed $\epsilon$-neighborhood of $G$. Here, $\displaystyle{\rho(x,G)=\inf_{y\in G}\rho(x,y)}$. If $G$ is measurable, we denote by $|G|$ its volume.\\

    The symmetric difference between two sets $G_1$ and $G_2$ is denoted by $G_1\triangle G_2$ and their Hausdorff distance is denoted and defined as:
    \begin{equation*}
        d_H(G_1,G_2)=\inf\{\epsilon>0:G_1\subseteq G_2^\epsilon, G_2\subseteq G_1^\epsilon\}.
    \end{equation*}
	We denote by $\K$ the class of all convex bodies in $\R_d$, and by $\K^1$ the collection of all those included in $B_d$. 
    The convex hull of $n$ i.i.d. random points $X_1,\ldots,X_n$ is denoted by $\hat K_n$. If $X_1,\ldots,X_n$ have a density $f$ with respect to the Lebesgue measure in $\R^d$, we denote by $\PP_f$ their joint probability measure and by $\E_f$ the corresponding expectation operator (we omit the dependency in $n$ unless stated otherwise). If $f$ is the uniform density on a convex body $K$, we rather use the notation $\PP_K$ and $\E_K$. In general, when the density $f$ of $X_1$ is supported on a convex body $K$, we denote by $d_f(K,\hat K_n)=\int_{K\setminus \hat K_n}f(x)\diff x$ and by $V_n$ the missing volume of $\hat K_n$, i.e., $V_n=|K|-|\hat K_n|$. The integral $d_f(K,\hat K_n)$ can be interpreted as a weighted missing volume. We are interested in deviation inequalities for $Z$, where $Z$ is either $d_f(K,\hat K_n)$ or $V_n$, i.e., in bounding from above $\DS \PP_K[Z>\epsilon]$, for $\epsilon>0$. We are also interested in upper bounds for the moments $\E_K[Z^q], q>0$. Our main result is stated in Section \ref{Sec:Gen}: We prove a deviation inequality for the weighted missing volume, and we investigate a special class of densities, satisfying the so called \textit{margin condition}, for which we are also able to control the unweighted missing volume. In Section \ref{Sec:Vert}, we investigate the moments of the number of vertices of the random polytope, with no restrictions on $K$ and on the underlying density on $K$, as long as it is bounded from above. Finally, in Section \ref{Sec:Unif}, we focus on the uniform case, and we derive a deviation inequality for the missing volume, and prove that the rates of the subsequent moment inequalities are tight. Last section is devoted to some proofs.

\section{Deviation inequality for the missing volume of random polytopes} \label{Sec:Gen}

Our main result is the following theorem.

\begin{theorem} \label{MainTheorem}
 
Let $n\geq 1$. Let $K\in\K^1$, $f$ be a density supported on $K$ and $X_1,\ldots,X_n$ be i.i.d. random points with density $f$. Assume that $f\leq M$ almost everywhere, for some positive number $M$. Then, there exist positive constants $C_1$ and $C_2$ that depend on $d$ only, such that the following holds.
\begin{align*}
	\PP\left[n(d_f(K,\hat K_n)-C_1(M+1)n^{-2/(d+1)}>x\right] \leq C_2e^{-x}, \quad \forall x\geq 0.
\end{align*}
	
\end{theorem}

\begin{proof}

This proof is inspired by Theorem 1 in \cite{KST}, which derives an upper bound for the risk of a convex hull type estimator of a convex function. It is based on an upper bound of the covering number of $\K^1$, proven by \cite{Bronshtein76}. For $\delta>0$, a $\delta$-net of $\K^1$ for the Hausdorff distance is a finite subset $\mathcal N_\delta$ of $\K^1$ such that for all $G\in\K^1$, there exists $G^*\in\mathcal N_\delta$ with $d_H(G,G^*)\leq \delta$. The covering number of $\K^1$ for the Hausdorff distance is the function that maps $\delta>0$ to the mimimum cardinality of a $\delta$-net of $\K^1$ for the Hausdorff distance. The following lemma is proven in \cite{Bronshtein76}. 

\begin{lemma} \label{LemmaBron}
	The covering number of $\K^1$ for the Hausdorff distance is not larger than $c_1 e^{\delta^{-(d-1)/2}}$, for all $\delta>0$, where $c_1>0$ depends on $d$ only.
\end{lemma}

Our next lemma shows that the Nykodim distance (i.e., the volume of the symmetric difference) between two sets in $\K^1$ is dominated by their Hausdorff distance. The proof is deferred to the appendix.

\begin{lemma}\label{Lemma2}
            There exists a positive constant $\alpha_1$ which depends on $d$ only, such that
            $$|G\triangle G'|\leq \alpha_1 d_H(G,G'), \quad \forall G,G'\in\K^1.$$
\end{lemma}
\vspace{2mm}
Let $\delta=n^{-2/(d+1)}$ and $\{K_1,\ldots,K_N\}$ be a $\delta$-net of $\K^1$, where $N$ is a positive integer satisfying $N\leq c_1 e^{\delta^{-(d-1)/2}}$, cf. Lemma \ref{LemmaBron}. Let $\hat j\leq N$ be such that $d_H(\hat K_n,K_{\hat j})\leq \delta$. By Lemma \ref{Lemma2}, this implies that $|\hat K_n\setminus K_{\hat j}|\leq |\hat K_n\triangle K_{\hat j}|\leq \alpha_1\delta$ and hence, since $f$ is nonnegative,
\begin{equation*}
	d_f(K,\hat K_n)\leq \int_{K\setminus K_{\hat j}}f+\int_{K_{\hat j}\setminus \hat K_n}f\leq \int_{K\setminus K_{\hat j}}f+\alpha_1 M\delta.
\end{equation*}
In addition, since $d_H(\hat K_n,K_{\hat j})\leq \delta$, it is true that $\hat K_n\subseteq K_{\hat j}^{\delta}$, yielding $X_i\in K_{\hat j}^\delta$, for all $i=1,\ldots,n$. Therefore, for all $\varepsilon\in (0,1)$, 
\begin{align}
	\PP_f\left[d_f(K,\hat K_n)>\varepsilon\right] & \leq \PP_f\left[d_f(K,K_{\hat j})>\varepsilon-\alpha_1 M\delta\right] \nonumber \\
	& \leq \PP_f\Big[\exists j=1,\ldots, N: d_f(K,K_j)>\varepsilon-M\delta \nonumber \\
	& \quad \quad \quad \quad \quad \quad \quad \quad \quad \mbox{ and } X_i\in K_j^{\delta}, \forall i=1,\ldots,n\Big] \nonumber \\
	& \leq \sum_{j\in I_{\varepsilon-\alpha_1 M\delta}} \PP_f\left[X_1\in K_j^\delta\right]^n \nonumber \\
	& = \sum_{j\in I_{\varepsilon-\alpha_1 M\delta}} \left(\int_{K_j^\delta}f\right)^n, \label{ProofMainStep1}
\end{align}
where we used the union bound, and for $\eta\in\R$, we denoted by $I_\eta=\{j=1,\ldots,N: d_f(K,K_j)>\eta\}$. Note that
\begin{equation} \label{ProofMainStep2}
	\int_{K_j^\delta}f = \int_{K_j}f + \int_{K_j^\delta \setminus K_j}f\leq 1-d_f(K,K_j) + |K_j^\delta \setminus K_j|M.
\end{equation}
By Lemma \ref{Lemma2}, the last term is bounded from above by $\alpha_1\delta$ and \eqref{ProofMainStep2} entails, if $j\in I_{\varepsilon-\alpha_1 M\delta}$,
\begin{equation*}
	\int_{K_j^\delta}f \leq 1-\varepsilon+2\alpha_1M\delta.
\end{equation*}
Hence, \eqref{ProofMainStep1} becomes
\begin{align}
	\PP_f\left[d_f(K,\hat K_n)>\varepsilon\right] & \leq N(1-\varepsilon+2\alpha_1M\delta)^n \nonumber \\
	& \leq c_1\exp\left(-n(\varepsilon-2\alpha_1 M\delta)+\delta^{-(d-1)/2}\right) \nonumber \\
	& = c_1\exp\left(-n(\varepsilon-(2\alpha_1 M+1)\delta)\right). \label{ProofMainStep3}
\end{align}
Note that since $d_f(K,\hat K_n)\leq 1$ almost surely, \eqref{ProofMainStep3} actually holds for all $\varepsilon>0$ (we have assumed $\varepsilon\in (0,1)$ so far). 
This ends the proof by taking $\varepsilon$ of the form $\DS \frac{x}{n}+(2\alpha_1 M+1)\delta$. 

\end{proof}

As a consequence of Theorem \ref{MainTheorem}, we get upper bounds for all the moments of $d_f(K,\hat K_n)$.

\begin{corollary} \label{MainCor}
	Let the assumptions of Theorem \ref{MainTheorem} hold. Then, for all $q>0$, there exists $A_q>0$ that depends on $q$ and $d$ only such that
	$$\E_f\left[d_f(K,\hat K_n)^q\right] \leq A_q(M+1)^{q}n^{-2q/(d+1)}.$$
\end{corollary}

\begin{proof}

The proof is based on an application of Fubini's theorem. Namely, if $Z$ is a nonnegative random variable and $q>0$, then
$$\E[Z^q]=q\int_0^\infty t^{q-1}\PP[Z>t]\diff t.$$
Let $Z=d_H(K,\hat K_n)$ and denote by $\delta=C_1(M+1)n^{-2/(d+1)}$. Then, 
\begin{align*}
	\E_f[Z^q] & = q\int_0^\infty t^{q-1}\PP[Z>t]\diff t \\
	& = q\int_0^\delta t^{q-1}\PP[Z>t]\diff t + q\int_\delta^\infty t^{q-1}\PP[Z>t]\diff t \\
	& \leq \delta^q + q\int_0^\infty (t+\delta)^{q-1}\PP[Z>t+\delta]\diff t \\
	& = \delta^q + \frac{q}{n}\int_0^\infty \left(\frac{x}{n}+\delta\right)^{q-1}\PP[Z>x/n+\delta]\diff x \\
	& \leq \delta^q + \frac{C_2 q}{n}\int_0^\infty \left(\frac{x}{n}+\delta\right)^{q-1}e^{-x}\diff x \\
	& \leq \delta^q + \frac{C_2 q}{n}\int_0^\infty 2^{q-2}\left(\frac{x^{q-1}}{n^{q-1}}+\delta^{q-1}\right) e^{-x}\diff x \\
	& \leq a_q\delta^q,
\end{align*}
for some constant $a_q$ that depends on $d$ and $q$ only. 

\end{proof}

The inequalities that we have obtained for $d_f(K,\hat K_n)$ can transfer to the missing volume $|K\setminus\hat K_n|$ under some conditions on $f$. An important such condition, called the \textit{margin condition} (see \cite{MammenTsybakov99,Tsybakov2004}), is the following. For $t>0$, let $K_f(t)=\{x\in K: f(x)\leq t\}$. The density $f$ satisfies the margin condition with parameters $\alpha\in (0,\infty]$, $L,t_0>0$ if and only if 
$$|K_f(t)|\leq Lt^\alpha, $$ 
for all $t\in (0,t_0]$. 
The case $\alpha=\infty$ corresponds to a density $f$ that is almost everywhere bounded away from zero on $K$. Let us give two other important cases where a margin condition is satisfied.
\vspace{3mm}

\textit{Slow decay of $f$ near the boundary of $K$:} Assume that $f$ does not decay too fast near the boundary of $K$ (which we denote by $\partial K$). Namely, assume the existence of positive numbers $\rho_0,c$ and $\gamma$ such that for all $x\in K$,
\begin{equation} \label{slowdecay}
	f(x)\geq c\min\left(\rho_0,\rho(x,\partial K)\right)^\gamma.
\end{equation}
Then, $f$ satisfies the margin condition with $t_0=c\rho_0^\gamma$, $L=\frac{\kappa}{c^{1/\gamma}}$ and $\alpha=1/\gamma$, where $\kappa$ is any number that no smaller than the surface area of $K$ (e.g., take $\kappa$ to be the surface area of the unit ball, if $K\in\K^1$).

\vspace{3mm}
\textit{Projection of higher dimensional convex bodies:} 
Let $D>d$ be an integer and $K_0\in\mathcal K_D^1$. Let $Y_1,\ldots,Y_n$ be i.i.d. uniform random points in $K_0$. Identify $\R^d$ with a linear subspace of $\R^D$ and let $H$ be its orthogonal space in $\R^D$. Let $K$ be the orthogonal projection of $K_0$ onto $\R^d$ and let $X_i$ be the orthogonal projection of $Y_i$ onto $\R^d$, for $i=1,\ldots,n$. Assume that $K_0$ satisfies the $r$-rolling ball condition, where $r>0$: Namely, assume that for all $x\in\partial K_0$, there exists $a\in K_0$ with $x\in B_D(a,r)\subseteq K_0$. Then, we have the following lemma, whose proof is deferred to the appendix.

\begin{lemma} \label{lemmaProj}
The density $f$ of the $X_i$'s satisfies \eqref{slowdecay}, with $\rho_0=r$, $c=r^{(D-d)/2}\beta_D\beta_{D-d}$ and $\gamma=(D-d)/2$.
\end{lemma}

Hence, as we already saw in the previous example, $f$ satisfies the margin condition with $\alpha=2/(D-d)$.

\vspace{4mm}

The following lemma gives a (deterministic) control of $d_f(K,\hat K_n)$ on the missing volume. For completeness of the presentation, we provide its proof in the appendix (see also Proposition 1 in \cite{Tsybakov2004}).

\begin{lemma} \label{LemmaMargin}
Let $f$ satisfy the margin condition with parameters $\alpha,L,t_0$. Assume that $d_f(K,\hat K_n)\leq t_0^{\alpha+1}$. Then,
$$|K\setminus \hat K_n|\leq (L+1)d_f(K,\hat K_n)^{\alpha/(\alpha+1)}.$$
\end{lemma}

If $f$ satisfies a margin condition, the deviation and moment inequalities that we have for $d_f(K,\hat K_n)$ transfer to the missing volume, as shown in the next two results.

\begin{theorem} \label{TheoremMargin}
	Let $f$ satisfy the margin condition with parameters $\alpha, L$ and $t_0$ and assume that $f\leq M$ for some positive number $M$. Then, there exists a positive integer $n_0$ that depends on $d$, $t_0$ and $M$ and positive constants $C_3$ and $C_4$ that depend on $d$, $\alpha$ and $L$ such that, for all $n\geq n_0$ and for all $x\geq 0$,
	
$$\PP_f\left[n^{\frac{\alpha}{\alpha+1}}\left(|K\setminus\hat K_n|-C_3n^{-\frac{2\alpha}{(\alpha+1)(d+1)}}\right)>x\right] \leq C_4e^{-x^{(\alpha+1)/\alpha}}+C_4e^{-nt_0/2}.$$
\end{theorem}

Note that the constant $C_4$ is the same as $C_2$ in Theorem \ref{MainTheorem} and that $n_0$ is the first integer $n$ that satisfies $C_1(M+1)n^{-2/(d+1)}\leq t_0/2$, where $C_1$ is defined in Theorem \ref{MainTheorem}.

\begin{proof}
For $\varepsilon>0$, write
\begin{align}
	& \PP_f\left[|K\setminus\hat K_n|>\varepsilon\right] \nonumber \\
	& =\PP_f\left[|K\setminus\hat K_n|>\varepsilon,d_f(K,\hat K_n)\leq t_0\right]+\PP_f\left[|K\setminus\hat K_n|>\varepsilon,d_f(K,\hat K_n)> t_0\right] \nonumber \\
	& \leq \PP_f\left[d_f(K,\hat K_n)>\varepsilon^{(\alpha+1)/\alpha}\right]+\PP_f\left[d_f(K,\hat K_n)> t_0\right]
\end{align}
and apply Theorem \ref{MainTheorem} to get the desired result.
\end{proof}

As a consequence of the deviation inequality of Theorem \ref{TheoremMargin}, we get the following moment inequalities.

\begin{corollary}
	Recall the notation and assumptions of Theorem \ref{TheoremMargin}. Then, for all $q>0$, there exists a positive constant $A_q'$ that depends on $d, \alpha, L, t_0$ and $q$ only, such that
	$$\E_f\left[|K\setminus \hat K_n|^q\right]\leq A_q'(M+1)^q n^{-\frac{2\alpha q}{(\alpha+1)(d+1)}}, \quad \forall n\geq n_0.$$
\end{corollary}

\begin{proof}
The proof is based on the same argument as in the proof of Corollary \ref{MainCor} and is omitted.
\end{proof}

It is easy to see that the constants $C_3$ and $C_4$ in Theorem \ref{TheoremMargin} are bounded, as functions of $\alpha$. Hence, when $\alpha=\infty$, which includes the case of the uniform distribution on $K$, the rate obtained in Theorem \ref{TheoremMargin} coincides with that obtained in Section \ref{Sec:Unif}. As a byproduct, the deviation inequality given in Theorem \ref{DevIneq} below, for the missing volume, still holds, with different constants, for any density that is bounded away from zero and infinity (see Remark \ref{RemarkNearlyUnif} in Section \ref{Sec:Unif}).

\section{Moment inequalities for the number of vertices of random polytopes} \label{Sec:Vert}

In this section, we are interested in the number of vertices of random polytopes. Let $\mu$ be any probability measure in $\R^d$. Let $X_1,X_2,\ldots$ be i.i.d. realizations of $\mu$ and $\hat K_n$ be the convex hull of the first $n$ of them, for $n\geq 1$. Denote by $\mathcal V_n$ the set of vertices of $\hat K_n$ and $R_n$ its cardinality, i.e., the number of vertices of $\hat K_n$. Efron \cite{Efron1965} proved a simple but elegant identity, which relates the expected missing mass $\E\left[1-\mu(\hat K_n)\right]$ to the expected number of vertices $\E[R_{n+1}]$ of $\hat K_{n+1}$. Namely, one has
\begin{equation} \label{EfronId}
	\E\left[1-\mu(\hat K_n)\right]=\frac{\E[R_{n+1}]}{n+1}, \forall n\in\N^*.
\end{equation}

In the case when $\mu$ is the uniform probability measure on a convex body $K$, extensions of this identity to higher moments of $|K\backslash\hat K_n|$ can be found in \cite{Buchta2005}. Here, we prove the following inequalities that hold for any distribution $\mu$.

%\begin{lemma} \label{EfronIdExt}
%For all positive integer $q$, 
%$$\frac{\E \left[R_{n+q}(R_{n+q}-1)\ldots(R_{n+q}-q+1)\right]}{(n+q)(n+q-1)\ldots(n+1)} \leq \E\left[\left(1-\mu(\hat K_n)\right)^q\right].$$
%\end{lemma}

\begin{lemma} \label{EfronIdExt}
For all positive integer $q$, 
$$\E \left[ \prod_{j=0}^{q-1} \frac{R_{n+q}-j}{n+q-j}\right] \leq \E\left[\left(1-\mu(\hat K_n)\right)^q\right].$$
\end{lemma}

If $\mu$ has a bounded density $f$ with respect to the Lebesgue measure and is supported on a convex body $K$, we can combining Corollary \ref{Corollary1} and Theorem \ref{EfronIdExt} yields the following inequality:

\begin{equation*}
	\E_f\left[R_{n}(R_{n}-1)\ldots(R_{n}-q+1)\right] \leq A_q(M+1)^q n^{\frac{q(d-1)}{d+1}}, \forall n\in\N^*, \forall q\in\N^*,
\end{equation*}
where $A_q$ is the same constant as in Corollary \ref{Corollary1} and $M$ is an almost everywhere upper bound of $f$. Since the polynomial $x^q$ is a linear combination of the polynomials $x(x-1)\ldots(x-k+1), 0\leq k\leq q$, we get the following result.

\begin{theorem} \label{ThmBoundRn}
Let $K$ be a convex body and $f$ satisfy the assumptions of Theorem \ref{MainTheorem}. Then, for all positive integer $q$, there exists a positive constant $B_q$ that depends on $d$ and $q$ only such that
\begin{equation}\label{UBeq}
	\E_f\left[R_n^q\right] \leq B_q(M+1)^q n^{\frac{q(d-1)}{d+1}}.
\end{equation}
\end{theorem}

\begin{remark}
	The boundedness assumption on $f$ in Theorem \ref{ThmBoundRn} seems unavoidable, in the following sense: Let $\mu$ be a probability measure that puts positive mass on arbitrarily many points in $K$, near its boundary and let $f$ be the density of a regularized version of $\mu$, truncated so it remains supported on $K$. Then, with high probability, $R_n$ can be arbitrarily large. 
\end{remark}

\section{The case of uniform distributions} \label{Sec:Unif}

As discussed in the introduction, the case of the uniform distribution on a convex body $K$ is an extremely important case in the stochastic geometry literature. In this section, we use the results proven in the previous sections in order to derive universal inequalities for the convex hull of uniform random points. By universal, we mean uniform for all convex bodies, irrespective of their volume, or facial structure.

\begin{theorem}\label{DevIneq}
	There exist two positive constants $\aleph_1, \aleph_2$ and $\aleph_3$, which depend on $d$ only, such that:
\begin{equation}
    \label{Theorem1} \sup_{K\in\K}\PP_K\left[n\left(\frac{|K\backslash\hat K_n|}{|K|}-\aleph_1 n^{-2/(d+1)}\right)>x\right] \leq \aleph_2 e^{-\aleph_3 x}, \forall x>0.
\end{equation}
\end{theorem}

\begin{proof}
In order to prove Theorem \ref{DevIneq}, we first state two lemmas, the first of which is about the so called \textit{John's ellipsoid} of a convex body.

\begin{lemma}\label{ellipsoid}
	For all $K\in\K$, there exists $a\in\R^d$ and an ellipsoid $E$ such that 
	\begin{equation} \label{Jellips} 
		a+d^{-1}(E-c)\subseteq K\subseteq E.
	\end{equation}
\end{lemma}
        Proof of Lemma \ref{ellipsoid} can be found in \cite{Leichtweiss59} and \cite{HugSchneider07}. If $E$ is an ellipsoid of maximum volume satisfying \eqref{Jellips} for some $a\in\R^d$, then $a+d^{-1}(E-c)$ is called \textit{John ellipsoid} of $K$.

Let $K\in\K$ and $X_1,\ldots,X_n$ be i.i.d. uniform random points in $K$. Let $E$ be an ellipsoid that satisfies \eqref{Jellips}, and $T$ an affine transform in $\R^d$ which maps $E$ to the unit ball $B_d$. Note that $TX_1,\ldots,TX_n$ are independent and uniformly distributed in $TK$, and their convex hull is $T\hat K_n$. Hence, the distribution of $\frac{|K\setminus\hat K_n|}{|K|}=\frac{|TK\setminus T\hat K_n|}{|TK|}$ is the same as that of $\frac{|TK\setminus\hat K'_n|}{|TK|}$, where $\hat K'_n$ is the convex hull of $X'_1,\ldots,X'_n$, which are i.i.d. uniform random points in $TK$. Therefore, for all $\varepsilon>0$,
\begin{equation}
	\PP_K\left[\frac{|K\setminus\hat K_n|}{|K|}>\varepsilon\right]=\PP_f[d_f(TK,\hat K_n)>\varepsilon],
\end{equation}
where $f$ is the uniform density on $TK$. The density $f$ is bounded from above by $M=1/|TK|\leq d^d/|TE|=d^d/\beta_d$, by definition of $E$ and $T$. Hence, applying Theorem \ref{MainTheorem} yields the desired result, since all the constants in that theorem depends on $d$ only.

\end{proof}

\begin{remark} \label{RemarkNearlyUnif}
	Similar arguments could be used to prove a deviation inequality with the same rate rate $n^{-2/(d+1)}$, for a density $f$ that is nearly uniform on $K$, i.e., that satisfies $0<m\leq f(x)\leq M$ for all $x\in K$, where $m$ and $M$ are positive numbers. Indeed, for all invertible affine transformation $T$, $d_g(TK,\hat K_n')$ and $d_f(K,\hat K_n)$ have the same distribution, where $g(y)=|\det T|^{-1}f(T^{-1}y), y\in TK$ and $\hat K_n'$ is the convex hull of $n$ i.i.d. points with density $g$. In addition, $\DS d_f(K,\hat K_n)\geq \frac{m}{M}\frac{|K\setminus\hat K_n|}{|K|}$. Therefore, the same reasoning as in the proof of Theorem \ref{DevIneq} yields
\begin{equation*}
	\PP_f\left[n\left(\frac{|K\backslash\hat K_n|}{|K|}-\aleph_1' n^{-2/(d+1)}\right)>x\right] \leq \aleph_2' e^{-\aleph_3' x}, \forall x>0,
\end{equation*}
where $\aleph_1',\aleph_2'$ and $\aleph_3'$ are positive constant that depend only on $d$ and on the ratio $M/m$.

\end{remark}

A drawback of Theorem \ref{DevIneq} is that it involves constants which depend at least exponentially on the dimension $d$. However, this seems to be the price for getting a uniform deviation inequality on $\K$. 

The following moment inequalities are a consequence of Theorem \ref{DevIneq}.

\begin{corollary}\label{Corollary1}
	For every positive number $q$, there exists a positive constant $A_q$, which depends on $d$ and $q$ only, such that
\begin{equation}
	\label{Cor1} \sup_{K\in\K}\E_K\left[\left(\frac{|K\backslash\hat K_n|}{|K|}\right)^q\right]\leq A_q n^{-2q/(d+1)}.
\end{equation}
\end{corollary}

Note that this corollary could also be derived from Vu's result \cite{Vu05}, combined with Giannopoulos and Tsolomitis' result \cite{GiannopoulosTsolomitis2003} (see Remark \ref{ImportantRemark} below).

Combining Corollary \ref{Corollary1} with the lower bound for the minimax risk obtained in \cite{Mammen95} yields the following result.

\begin{corollary}\label{Corollary2}
	For every positive number $q$, there exist positive constants $a_q$ and $A_q$, which depend on $d$ and $q$ only, such that
	$$a_qn^{-2q/(d+1)} \leq \sup_{K\in\K}\E_K\left[\left(\frac{|K\backslash\hat K_n|}{|K|}\right)^q\right]\leq A_q n^{-2q/(d+1)}.$$
\end{corollary}

\begin{remark} \label{ImportantRemark}
	An analogous result to Theorem \ref{DevIneq} could be derived from Vu's result \eqref{DevVu} combined with an elegant result proven by Giannopoulos and Tsolomitis \cite{GiannopoulosTsolomitis2003}, Theorem 3.6. Let $K\in\K^d$ have volume one and $\phi$ be a non decreasing function defined on the positive real line. Then, the expectation $\E_K[\phi(|\hat K_n|)]$ is minimized when $K$ is an ellipsoid. The key argument is that this expectation does not increase when $K$ is replaced by its Steiner symmetral with respect to a hyperplane, and performing such a transform iteratively on $K$ leads to a Euclidean ball at the limit. When $\phi$ is the indicator function of the interval $(x,\infty)$, for $x>0$, Giannopoulos and Tsolomitis' results implies that $\displaystyle{\PP\left[V_n\geq x\right]}$ is maximized when $K$ is an ellipsoid. Hence, applying \eqref{DevVu} to an ellipsoid of volume one yields to a uniform deviation inequality as in Theorem \ref{DevIneq}. However, the range for $x$ would be much smaller than ours, which allows to capture the whole right tail of $V_n$. Yet, it would still yield similar bounds for the moments of $V_n$, as in Corollary \ref{Corollary1}.

\end{remark}

Let $K\in\K$ and $n, q$ be positive integers. Hölder inequality yields $\E_K\left[R_n^q\right]\geq \E_K\left[R_n\right]^q$ and, by Efron's identity \eqref{EfronId}, 
\begin{equation*}
	\E_K\left[R_n^q\right]\geq n^q\E_K\left[\frac{|K\backslash\hat K_{n-1}|}{|K|}\right]^q. 
\end{equation*}
%If the affine surface area of $K$ is positive, which occurs, for instance, when the boundary of $K$ is smooth with positive Gauss curvature, then it is known that $\E_K\left[\frac{|K\backslash\hat K_{n-1}|}{|K|}\right]$ is exactly of the order of $n^{-2/(d+1)}$ (see \cite{Schutt93}). It is shown in \cite{Groemer74} that $\E_K\left[\frac{|K\backslash\hat K_{n-1}|}{|K|}\right]$ is maximal when $K$ is an ellipsoid, and and the corresponding value is exactly of the order of $n^{-2/(d+1)}$, up to a constant factor which depends on the dimension $d$ only. Therefore, for smooth convex bodies $K$ with positive curvature, the rate of the upper bound in \eqref{UBeq} is tight, and one gets the following theorem:
Hence, one gets the following theorem:

\begin{theorem} \label{UB}
		Let $n$ and $q$ be positive integers. Then, for some positive constants $b_q$ and $B_q$ which depend on $d$ and $q$ only,
		\begin{equation*}
			b_q n^{\frac{q(d-1)}{d+1}} \leq \sup_{K\in\mathcal K_d}\E_K\left[R_n^q\right] \leq B_q n^{\frac{q(d-1)}{d+1}}.
		\end{equation*}
\end{theorem}

Combined with Theorem \ref{ThmBoundRn}, this theorem has two consequences. First, the rate in the upper bound of Theorem \ref{ThmBoundRn} is tight, uniformly on all arbitrary convex bodies and bounded densities. Second, the uniform case is the worst case, i.e., yields the largest possible rate for the expected number of vertices of $\hat K_n$. Namely, the following holds. For $K\in\K^1$ and $M>0$, denote by $\mathcal F(K,M)$ the collection of all densities that are supported on $K$ and bounded by $M$. For two positive sequences $u_n$ and $v_n$, write $u_n=O(v_n)$ if the ratio $u_n/v_n$ is bounded, independently of $n$.

\begin{theorem} \label{ThmOpt}
	For all $M>0$ and all positive integer $q$,
\begin{equation*}
	\sup_{K\in\K^1}\sup_{f\in\mathcal F(K,M)} \E_f\left[R_n^q\right] = O\left(\sup_{K\in\K^1}\E_K\left[R_n^q\right]\right).
\end{equation*}
\end{theorem}

\begin{remark}
We do not know whether the supremum over $K$ could be removed in Theorem \ref{ThmOpt}: We propose the following open question. Is it true that for all $M>0$ and $K\in\K^1$,
\begin{equation*}
	\sup_{f\in\mathcal F(K,M)} \E_f\left[R_n\right] = O\left(\E_K\left[R_n\right]\right) \quad ?
\end{equation*}

\end{remark}

\section{Appendix: Proof of the lemmas}

\paragraph{Proof of Lemma \ref{Lemma2}:}

 Let $G\in\K$. Steiner formula (see Section 4.1 in \cite{SchneiderBook}) states that there exist positive numbers $L_1(G),\ldots,L_d(G)$, such that
            \begin{equation}
                \label{SteinerFormula}|G^\lambda\backslash G|=\sum_{j=1}^d L_j(G)\lambda^j, \lambda\geq0.
            \end{equation}
            Besides the $L_j(G), j=1,\ldots,d$ are increasing functions of $G$. In particular, if $G\in\K^1$, then $L_j(G)\leq L_j(B_d)$. \\
            Let $G,G'\in\K^1$, and let $\lambda=d_H(G,G')$. Since $G$ and $G'$ are included in the unit ball, $\lambda$ is not greater than its diameter, so $\lambda\leq 2$. By definition of the Hausdorff distance, $G\subseteq G'^{\lambda}$ and $G'\subseteq G^{\lambda}$. Hence,
            \begin{align*}
                |G\triangle G'| & = |G\backslash G'|+|G'\backslash G| \leq |G'^{\lambda}\backslash G'|+|G^{\lambda}\backslash G| \\
                & \leq 2\sum_{j=1}^d L_j(B_d)\lambda^j \leq \lambda\sum_{j=1}^d L_j(B_d)2^j.
            \end{align*}
            The Lemma is proved by setting $\alpha_1=\sum_{j=1}^d L_j(B_d)2^j$.

        Note that since $\delta\leq 1$, Steiner formula \eqref{SteinerFormula} implies, for $G\in\K^1$, that
        \begin{equation}
            \label{Steiner}|G^\delta\backslash G|\leq \alpha_2\delta,
        \end{equation}
        where $\alpha_2=\sum_{j=1}^d L_j(B_d)$. \hfill \textsquare

\paragraph{Proof of Lemma \ref{lemmaProj}:}

For $x\in K$, 
$$f(x)=\frac{\textsf{Vol}_{D-d}\left((x+H)\cap K_0\right)}{\textsf{Vol}_D(K_0)},$$
where, for all integers $p$, $\textsf{Vol}_p$ stands for the $p$-dimensional volume and we recall that $H$ is the orthogonal space of $\R^d$ in $\R^D$. Let $x\in K$ with $t=\rho(x,\partial K)\leq r$. Let $x'\in\partial K$ such that $\rho(x,\partial K)=\rho(x,x')$. Let $x_0\in\partial K_0$ whose orthogonal projection onto $\R^d$ is $x'$. By the $r$-rolling condition, there exists $a\in K_0$ with $x'\in B_D(a,r)\subseteq K_0$. Note that $x'-a\in\R^d$ ($\R^d$ being identified to a subspace of $R^D$, orthogonal to $H$) since the (unique) tangent space to $K_0$ at $x'$ needs to be tangent to $B_D(a,r)$ as well. Therefore, 
\begin{equation*}
	\textsf{Vol}_{D-d}\left((x+H)\cap K_0\right) \geq \textsf{Vol}_{D-d}\left((x+H)\cap B_D(a,r)\right)
\end{equation*}	
and $(x+H)\cap B_D(a,r)$ is a $(D-d)$-dimensional ball with radius $h$, where $h=\sqrt{2rt-t^2}\geq \sqrt{rt}$. Hence, for all $x\in K$ with $\rho(x,\partial K)\leq r$,
\begin{equation*}
f(x)\geq \frac{(rt)^{(D-d)/2}\beta_{D-d}}{\textsf{Vol}_D(K_0)} \geq (rt)^{(D-d)/2}\beta_{D-d}\beta_D,
\end{equation*}
which proves the lemma. \hfill \textsquare

\paragraph{Proof of Lemma \ref{LemmaMargin}:}

For all $t\in (0,t_0]$,
\begin{align}
	|K\setminus\hat K_n| & = \int_K \mathds 1_{x\notin \hat K_n}\diff x \nonumber \\
	& = \int_{K_f(t)} \mathds 1_{x\notin \hat K_n}\diff x + \int_{K\setminus K_f(t)} \mathds 1_{x\notin \hat K_n}\diff x \nonumber \\
	& \leq |K_f(t)|+\frac{1}{t}\int_{K\setminus K_f(t)} \mathds f(x) 1_{x\notin \hat K_n}\diff x \nonumber \\
	& \leq Lt^\alpha+\frac{1}{t}d_f(K,\hat K_n). \label{LemmaMargin1}
\end{align}
If $d_f(K,\hat K_n)\leq t_0^{\alpha+1}$, take $t=d_f(K,\hat K_n)^{\alpha+1}$ in \eqref{LemmaMargin1}. \hfill \textsquare

\paragraph{Proof of Lemma \ref{EfronIdExt}:}

For precision's sake, we denote by $\PP^{\otimes n}$ the $n$-product of the probability measure $\mu$, i.e., the joint probability measure of the random variables $X_1,\ldots,X_n$, and by $\E^{\otimes n}$ the corresponding expectation operator. First, note that the expectation $\E^{\otimes n}[(1-\mu(\hat K_n))^q]$ can be rewritten as :

\begin{align}
	\E^{\otimes n}\left[(1-\mu(\hat K_n))^q\right] & = \E^{\otimes n}\left[\PP^{\otimes q}\left[X_{n+1}\notin \hat K_n, \ldots, X_{n+q}\notin \hat K_n|X_1,\ldots,X_n\right]\right] \nonumber \\
	& = \PP^{\otimes (n+q)}\left[X_{n+j}\notin\hat K_n, \forall j=1,\ldots,q\right]. \label{step01}
\end{align}
Using the symmetric role of $X_1,\ldots,X_{n+q}$, and since the event $\{X_{n+j}\notin\hat K_n, \forall j=1,\ldots,q\}$ contains the event $\{X_{n+j}\in\mathcal V_{n+q}, \forall j=1,\ldots,q\}$, \eqref{step01} yields
\begin{align}
	\E^{\otimes n} & \left[(1-\mu(\hat K_n))^q\right] \nonumber \\
	& \geq \PP^{\otimes n+q}\left[X_{n+j}\in\mathcal V_{n+q}, \forall j=1,\ldots,q\right] \nonumber \\[2mm]
	& = \frac{1}{{{n+q}\choose{q}}}\sum_{1\leq i_1<\ldots<i_q\leq n+q}\PP^{\otimes n+q}\left[X_{i_j}\in\mathcal V_{n+q}, \forall j=1,\ldots,q\right] \nonumber \\[2mm]
	& = \frac{1}{{{n+q}\choose{q}}}\E^{\otimes n+q}\left[\sum_{1\leq i_1<\ldots<i_q\leq n+q} \mathds 1\left(X_{i_j}\in\mathcal V_{n+q}, \forall j=1,\ldots,q\right)\right] \nonumber \\[2mm]
	& = \frac{1}{{{n+q}\choose{q}}}\E^{\otimes n+q}\left[{R_{n+q}}\choose{q}\right] \nonumber \\[2mm]
	& = \frac{\E^{\otimes n+q}\left[R_{n+q}(R_{n+q}-1)\ldots(R_{n+q}-q+1)\right]}{(n+q)(n+q-1)\ldots(n+1)}, \nonumber
\end{align}
which proves the lemma. \hfill \textsquare

\bibliographystyle{plain}
\bibliography{Biblio}

\begin{thebibliography}{10}

\bibitem{AffentrangerWieacker91}
F.~Affentranger and J.~A. Wieacker.
\newblock On the convex hull of uniform random points in a simple $d$-polytope.
\newblock {\em Discrete Comput. Geom.}, 6:291--305, 1991.

\bibitem{BaranyBuchta93}
I.~B\'ar\'any and C.~Buchta.
\newblock Random polytopes in a convex polytope, independence of shape, and
  concentration of vertices.
\newblock {\em Math. Ann.}, 297:467--497, 1993.

\bibitem{BaranyLarman88}
I.~B\'ar\'any and D.~G. Larman.
\newblock Convex bodies, economic cap coverings, random polytopes.
\newblock {\em Mathematika}, 35:274--291, 1988.

\bibitem{Bronshtein76}
E.~M. Bronshtein.
\newblock $\epsilon$-entropy of convex sets and functions.
\newblock {\em Siberian Mathematical Journal}, 17:393--398, 1976.

\bibitem{Buchta2005}
C.~Buchta.
\newblock An identity relating moments of functionals of convex hulls.
\newblock {\em Discrete Comput Geom}, 33:125--142, 2005.

\bibitem{BuchtaMuller84}
C.~Buchta and J.~Müller.
\newblock Random polytopes in a ball.
\newblock {\em Journal Appl. Prob.}, 21:753--762, 1984.

\bibitem{Efron1965}
B.~Efron.
\newblock The convex hull of a random set of points.
\newblock {\em Biometrika}, 52:331--343, 1965.

\bibitem{GiannopoulosTsolomitis2003}
A~Giannopoulos and A.~Tsolomitis.
\newblock Volume radius of a random polytope in a convex body.
\newblock {\em Math. Proc. Cambridge Philos.}, 134:13--21, 2003.

\bibitem{Groemer74}
H.~Groemer.
\newblock On the mean value of the volume of a random polytope in a convex set.
\newblock {\em Arch. Math.}, 25:86--90, 1974.

\bibitem{HugSchneider07}
D.~Hug and R.~Schneider.
\newblock Stability result for a volume ratio.
\newblock {\em Israel J. Math.}, 161:209--219, 2007.

\bibitem{KST}
A.~P. Korostelev, L.~Simar, and A.~B. Tsybakov.
\newblock On estimation of monotone and convex boundaries.
\newblock {\em Publications de l'Institut de Statistique de l'Université de
  Paris}, 39:3--18, 1995.

\bibitem{Leichtweiss59}
K.~Leichtweiß.
\newblock Über die affine {Exzentrizität} konvexer {Körper}.
\newblock {\em Arch. Math.}, 10:187--199, 1959.

\bibitem{Mammen95}
E.~Mammen and A.~Tsybakov.
\newblock Asymptotical {Minimax} {Recovery} of {Sets} with {Smooth}
  {Boundaries}.
\newblock {\em Annals of Statistics}, 23:502--524, 1995.

\bibitem{MammenTsybakov99}
E.~Mammen and A.~Tsybakov.
\newblock Smooth discriminant analysis.
\newblock {\em Annals of Statistics}, 27:1808--1829, 1999.

\bibitem{Reitzner03}
M.~Reitzner.
\newblock Random polytopes and the {Efron}-{Stein} jackknife inequality.
\newblock {\em The Annals of Probability}, 31:2136--2166, 2003.

\bibitem{RenyiSulanke63}
A.~Rényi and R.~Sulanke.
\newblock Über die konvexe {Hülle} von $n$ zufällig gewählten {Punkten}.
\newblock {\em Z. Wahrscheinlichkeitsth. Verw. Geb.}, 2:75--84, 1963.

\bibitem{RenyiSulanke64}
A.~Rényi and R.~Sulanke.
\newblock Über die konvexe {Hülle} von $n$ zufällig gewählten {Punkten}.
  ii.
\newblock {\em Z. Wahrscheinlichkeitsth. Verw. Geb.}, 3:138--147, 1964.

\bibitem{SchneiderBook}
R.~Schneider.
\newblock {\em Convex bodies: the {Brunn}-{Minkowski} theory}.
\newblock Cambridge University Press, 1993.

\bibitem{Schutt93}
C.~Schütt.
\newblock On the affine surface area.
\newblock {\em Proceedings of the American Mathematical Society},
  118:1213--1218, 1993.

\bibitem{SchuttWerner90}
C.~Schütt and E.~Werner.
\newblock The convex floating body.
\newblock {\em Math. Scand}, 66:275--290, 1990.

\bibitem{Tsybakov2004}
A.~Tsybakov.
\newblock Optimal aggregation of classifiers in statistical learning.
\newblock {\em Annals of Statistics}, 32:135--166, 2004.

\bibitem{Vu05}
V.~H. Vu.
\newblock Sharp concentration of random polytopes.
\newblock {\em Geom. Funct. Anal.}, 15:1284--1318, 2005.

\end{thebibliography}

\end{document}